\documentclass[reqno]{amsart}
\usepackage{hyperref}

\begin{document}
\title[\hfilneg  \hfil Blow-up of solutions for Emden-Fowler type wave equation]
{Blow up at finite time for wave equation in viscoelasticity: a new kind for one spatial variable Emden-Fowler type }

\author[L. K. Laouar, Kh. zennirand A. Guesmia \hfil  \hfilneg]
{Lakhdar Kassah Laouar, Khaled zennir and Amar Guesmia}

\address{Lakhdar Kassah Laouar \newline
	Departement de Mathematiquess, Universit\'e de Constantine 1, Algerie.}
\email{l.kassah@umc.edu.dz} 

\address{Khaled zennir \newline
	First address: Department of Mathematics, College of
	Sciences and Arts, Al-Ras, Qassim University, Kingdom of Saudi
	Arabia.\newline
	Second address:  Laboratory LAMAHIS, Department of mathematics, University 20 Ao\^ut 1955- Skikda, 21000, Algeria}
\email{k.zennir@qu.edu.sa} 
\address{Amar Guesmia \newline
	Laboratoire de LAMAHIS, Departement de mathematiques, Universit\'e 20 Ao\^ut 1955- Skikda, 21000, Algerie}
\email{guesmiasaid@yahoo.fr}
\subjclass[2010]{35B44,35D30, 35L05}
\keywords{Local solution; Viscoelastic, Emden-Fowler wave equation; blow-up}

\begin{abstract}  For one spatial variable, a new kind of nonlinear wave equation for Emden-Fowler type is considered with boundary value null and initial values. Under certain conditions on the initial data and the exponent $p$, we exhibit that the viscoelastic term leads our problem to be dissipative and the global solutions still non-exist in $L^2$ at given finite time.  
\end{abstract}

\maketitle
\numberwithin{equation}{section}
\newtheorem{theorem}{Theorem}[section]
\newtheorem{lemma}[theorem]{Lemma}
\allowdisplaybreaks

\section{Introduction  }
We  consider a new kind of Emden-Fowler type wave equation in viscoelasticity 
\begin{equation} \left\{
\begin{array}{ll}\label{e1}
t^2u''-  u_{xx} +\int_{1}^{t}\mu\left(s\right) u_{xx}\left(t-s\right) ds =u^p\quad \hbox{ in }[1,T)  \times(r_1,r_2),\\
u(1,x) =u_0(x) \in H^2(r_1,r_2)  \cap H_0^{1}(r_1,r_2),\\
u'(1,x) =u_1(x) \in H_0^{1}(r_1,r_2) 
\end{array}
\right.
\end{equation} 
where $p>1$, $r_1$ and $r_2$ are real numbers and the scalar function $\mu$ (so-called relaxation kernel) is assumed to only be $\mu: \mathbb{R}^+\rightarrow \mathbb{R}^+$ of $C^1$, nonincreasing and satisfying
 \begin{eqnarray}
 \mu(0)>0, 1-\int_{0}^{\infty}e^{s/2}\mu(s)ds=l>0.\label{2.1}
 \end{eqnarray} 
 The study of the Emden-Fowler equation originated from earlier theories concerning gaseous dynamics in astrophysics around the turn of the 20-th century. The fundamental problem in the study of stellar structure at that time was to study the equilibrium configuration of the mass of spherical clouds of gas. The Emden-Fowler equation has an impact on many astrophysics evolution phenomena. It has been poorly studied by scientists until now, essentially in the qualitative point of view. \\
Under the assumption
that the gaseous cloud is under convective equilibrium (first proposed in 1862
by Lord Kelvin \cite{T1}), Lane studied the equation
\begin{equation}
\frac{d}{dt}\Big(t^2\frac{du}{dt}\Big)  +t^2u^p=0, \label{e*}
\end{equation}
for the cases $p=1.5$ and $2.5$. Equation (\ref{e*}) is commonly referred
to as the Lane-Emden equation \cite{C1}. Astrophysicists
were interested in the behavior of the solutions of
(\ref{e*}) which satisfy the initial condition: $u(0)  =1$,
$u'(0)  =0$. Special cases of (\ref{e*}), namely, when
$p=1$ the explicit solution to
\[
\frac{d}{dt}\big(t^2\frac{du}{dt}\big)  +t^2u=0,\quad
u(0)  =1,\; u'(0)=0\]
is $u=\sin(t)/t$, and when $p=5$, the explicit solution to
\[
\frac{d}{dt}\big(t^2\frac{du}{dt}\big)  +t^2u^{5}=0,\quad
u(0)  =1,\; u'(0)=0
\]
is $u=1/\sqrt{1+t^2/3}$.\\
Many properties of solutions to the Lane-Emden equation were studied by Ritter
\cite{R2} in a series of 18 papers published during 1878-1889. The publication of
Emden's treatise Gaskugeln \cite{E1} marks the end of first epoch in the study of
stellar configurations governed by (\ref{e*}). The
mathematical foundation for the study of such an equation and also of the more
general equation
\begin{equation}
\frac{d}{dt}\Big(t^{\rho}\frac{du}{dt}\Big)  +t^{\sigma}u^{\gamma}=0,\quad
t\geq0, \label{e**}
\end{equation}
was made by   Fowler \cite{F1,F2,F3,F4} in a series of four papers during 1914-1931.\\ The first serious
study on the generalized Emden-Fowler equation
\[
\frac{d^2u}{dt^2}+a(t)  | u| ^{\gamma}\operatorname{sgn}u=0,\quad t\geq0,
\]
was made by Atkinson and \textit{al}. \\ Recently, M.-R. Li in \cite{L11} considered and studied the blow-up phenomena of solutions to
the Emden-Fowler type semilinear wave equation
$$
t^2u_{tt}-u_{xx}=u^p\quad \hbox{ in }[1,T)  \times(a,b)).
$$\\
The present research aims to extend the study of mden-Fowler type wave equation to the case when the viscoelastic term is injected in domain $[r_1,r_2]$ where there
is no result about this topic. Thus, a wider class of phenomena can be modeled. \\
The main results here are to exhibit the role of the viscoelasticity, which makes our problem \eqref{e1} dissipative, in the Blow up of solutions in $L^2$ at finite time given by $$\ln T_1^{\ast}, s.t. T_1^{\ast}=\frac{2}{p-1}T_1^{\ast}=\frac{2}{p-1}\Big( \int_{r_1}^{r_2} \vert u_0\vert dx\Big) \Big(\int_{r_1}^{r_2} u_0u_1dx\Big)^{-1},$$ for Emden-Fowler type wave equation when the energy is null which will be the main results of subsection 3.1. In the subsection 3.2, we will discuss the blow up in finite time $\ln T_2^{\ast}<\ln T_1^{\ast}$ of problem \eqref{e1} for large class of solution in the case when the associated energy is negative. The questions of local existence and uniqueness will be also considered in the section 2.\\
\section{Preliminaries, local Existence of unique solution}
 
Under some suitable transformations, we can get the local existence of solutions to equation (\ref{e1}). 
Taking the transform
$$\tau=\ln t, \qquad v   =u, \qquad u_{xx}=v_{xx},$$ then $$u'=t^{-1}v_{\tau},\qquad
t^2 u''=-v_{\tau}+v_{\tau\tau},$$ equation (\ref{e1}) takes the form
\begin{eqnarray} \label{e4}
&&v_{\tau\tau}- v_{xx} +\int_{0}^{\tau}\mu(s) v_{xx}(\tau-s) ds =v_{\tau}+v^p\quad \hbox{ in }[0,\ln T)  \times(r_1,r_2), \nonumber\\\nonumber\\
&&v(x,0) =u_0(x),\quad u_{\tau}(x,0) =u_1(x), \nonumber\\\nonumber\\
&& v(r_1,\tau) = v(r_2,\tau)=0.
\end{eqnarray}
Let
\begin{eqnarray}
&&v(\tau,x)  =e^{\tau/2}w(\tau,x),\nonumber \\\nonumber\\
&&v_{\tau}(\tau,x)=e^{\tau/2}w_{\tau}(\tau,x)+\frac{1}{2}e^{\tau/2}w(\tau,x),\nonumber\\\nonumber\\
&&v_{\tau\tau}(\tau,x)=e^{\tau/2}w_{\tau\tau}(\tau,x)+e^{\tau/2}w_{\tau}(\tau,x)+\frac{1}{4}e^{\tau/2}w(\tau,x),\nonumber
\end{eqnarray}
then (\ref{e4}) can be rewritten as
\begin{eqnarray}  
&&e^{\tau/2}w_{\tau\tau}-e^{\tau/2}w_{xx}+\int_{0}^{\tau}e^{s/2}\mu(s)w_{xx}(\tau-s)ds,\nonumber\\\nonumber\\
&& =\frac{1}{4}e^{\tau/2}w+e^{p\tau/2}w^p, \nonumber 
\end{eqnarray}
then
\begin{eqnarray} \label{e5}
w_{\tau\tau}- w_{xx} +e^{-\tau/2}\int_{0}^{\tau}e^{s/2}\mu(s) w_{xx}(\tau-s) ds =\frac{1}{4}w+e^{(p-1)\tau/2}w^p.
\end{eqnarray}
The following technical Lemma will play an important role.
\begin{lemma} \label{lemma0}
	For any $w\in C^{1}\left( 0,T,H^{1}(r_1,r_2 )\right)$ we have for any nonincreasing differentiable function $\alpha$ satisfying $\alpha(\tau)>0$
	\begin{eqnarray}
	&&\int_{r_1}^{r_2}\alpha(\tau)\int_{0}^{\tau}e^{s/2}\mu(\tau-s) w_{xx}(s) w'(\tau)dsdx\nonumber\\\nonumber\\
	&=&\frac{1}{2}\frac{d}{d\tau}\alpha(\tau)  \int_{0}^{\tau} e^{s/2}\mu(\tau-s)\int_{r_1}^{r_2} \vert w_x(\tau)-w_x(s)\vert^2dxds \nonumber\\\nonumber\\
	&-&\frac{1}{2}\frac{d}{d\tau}\alpha(\tau) \int_{0}^{\tau}e^{s/2}\mu(s)ds\int_{r_1}^{r_2} \left\vert w_{x}(\tau)\right\vert ^{2}dx  \nonumber\\\nonumber\\
	&-&\frac{1}{2}\alpha\int_{0}^{\tau} \Big(e^{s/2}\mu(\tau-s)\Big)'\int_{r_1}^{r_2} \vert w_x(\tau)-w_x(s)\vert^2dxds\nonumber\\\nonumber\\
	&+&\frac{1}{2}\alpha(\tau)e^{\tau/2}\mu(\tau)\int_{r_1}^{r_2} \left\vert w_{x}(\tau)\right\vert ^{2}dx \nonumber\\\nonumber\\
	&&-\frac{1}{2}\alpha'(\tau)\int_{0}^{\tau} e^{s/2}\mu(\tau-s)\int_{r_1}^{r_2} \vert w_x(\tau)-w_x(s)\vert^2dxds\nonumber\\ \nonumber\\
	&+&\frac{1}{2}\alpha'(\tau)\int_{0}^{s}e^{s/2}\mu(s)ds\int_{r_1}^{r_2} \left\vert w_{x}(\tau)\right\vert ^{2}dx.\nonumber
	\end{eqnarray}
\end{lemma}
\begin{proof} It's not hard to see
	\begin{eqnarray}
	&&\int_{r_1}^{r_2}\alpha(\tau)\int_{0}^{\tau}e^{s/2}\mu(\tau-s) w_{xx}(s) w'(\tau)dsdx\nonumber\\\nonumber\\
	&=&-\alpha(\tau)\int_{0}^{\tau}e^{s/2}\mu(\tau-s)\int_{r_1}^{r_2}  w'_x(\tau) w_{x}(s)dxds \nonumber \\\nonumber\\
	&=&-\alpha(\tau)\int_{0}^{\tau}e^{s/2}\mu(\tau-s)\int_{r_1}^{r_2} w'_x(v)\left[   w_{x}(s)- w_{x}(\tau)\right] dxds \nonumber \\\nonumber\\
	&&-\alpha(\tau)\int_{0}^{\tau}e^{s/2}\mu(s)\int_{r_1}^{r_2}w'_x(\tau)  w_{x}(\tau)dxds.\nonumber
	\end{eqnarray}
	Consequently,
	\begin{eqnarray}
	&&\int_{r_1}^{r_2}\alpha(\tau)\int_{0}^{\tau}e^{s/2}\mu(\tau-s) w_{xx}(s) w'(\tau)dsdx\nonumber\\\nonumber\\
	&=&\frac{1}{2}\alpha(\tau)\int_{0}^{\tau}e^{s/2}\mu(\tau-s)\frac{d}{d\tau}\int_{r_1}^{r_2} \left\vert   w_{x}(s)- w_{x}(\tau)\right\vert ^{2}dxds \nonumber\\\nonumber\\
	&&-\alpha(\tau)\int_{0}^{\tau}e^{s/2}\mu(s)\left( \frac{d}{d\tau}\frac{1}{2}\int_{r_1}^{r_2} \left\vert w_{x}(\tau)\right\vert ^{2}dx\right) ds\nonumber
	\end{eqnarray}
	which implies,
	\begin{eqnarray}
	&&\int_{r_1}^{r_2}\alpha(\tau)\int_{0}^{\tau}e^{s/2}\mu(\tau-s) w_{xx}(s) w'(\tau)dsdx\nonumber\\\nonumber\\
	&=&\frac{1}{2}\frac{d}{d\tau}\left[\alpha(\tau) \int_{0}^{\tau}e^{s/2}\mu(\tau-s)\int_{r_1}^{r_2} \left\vert   w_{x}(s)- w_{x}(\tau)\right\vert ^{2}dxds\right] \nonumber\\\nonumber\\
	&&-\frac{1}{2}\frac{d}{d\tau}\left[\alpha(\tau)\int_{0}^{\tau}e^{s/2}\mu(s)\int_{r_1}^{r_2} \left\vert w_{x}(v)\right\vert ^{2}dxds\right] \nonumber\\\nonumber\\
	&&-\frac{1}{2}\alpha(\tau)\int_{0}^{\tau}\Big(e^{s/2}\mu(\tau-s)\Big)'\int_{r_1}^{r_2} \left\vert   w_{x}(s)- w_{x}(\tau)\right\vert ^{2}dxds \nonumber\\\nonumber\\
	&&+\frac{1}{2}\alpha(\tau)e^{\tau/2}\mu(\tau)\int_{r_1}^{r_2} \left\vert w_{x}(\tau)\right\vert ^{2}dx.\nonumber\\\nonumber\\
	&&-\frac{1}{2}\alpha'(\tau)e^{s/2}\mu(\tau-s)\int_{r_1}^{r_2} \left\vert   w_{x}(s)- w_{x}(\tau)\right\vert ^{2}dxds\nonumber\\\nonumber\\
	&&+\frac{1}{2}\alpha'(\tau)\int_{0}^{s}e^{s/2}\mu(s)ds\int_{r_1}^{r_2} \left\vert w_{x}(\tau)\right\vert ^{2}dxds.\nonumber
	\end{eqnarray}
	This completes the proof.
\end{proof}
We introduce the modified energy associated to problem (\ref{e5}).
\begin{eqnarray}
&&2E_{w}(\tau) = \int_{r_1}^{r_2} \vert w_{\tau}\vert^2dx+(1-\int_{0}^{\tau}e^{s/2}\mu(s)ds) \int_{r_1}^{r_2} \vert w_{x}\vert^2dxd \nonumber\\\nonumber\\
&&+\int_{0}^{\tau}e^{s/2}\mu(\tau-s)\int_{r_1}^{r_2} \left\vert   w_{x}(s)- w_{x}(\tau)\right\vert ^{2}dxds\nonumber\\\nonumber\\
&&-\frac{1}{4}\int_{r_1}^{r_2} \vert w\vert^2dx-\frac{2}{p+1}e^{\frac{(p-1)\tau}{2}}\int_{r_1}^{r_2}\vert w\vert^{p+1}dx.\label{energy}
\end{eqnarray} 
and 
\begin{eqnarray}
&&2E_{w}(0) = \int_{r_1}^{r_2} ( u_{1	}-\frac{1}{2}u_0)^2dx+\int_{r_1}^{r_2} \vert u_{0x}\vert^2dx\nonumber\\\nonumber\\
&&+\int_{r_1}^{r_2}u_{0}u_1dx-\frac{2}{p+1}\int_{r_1}^{r_2}\vert u_0\vert^{p+1}dx.\nonumber
\end{eqnarray} 
Direct differentiation, using (\ref{2.1}), (\ref{e5}), leads to
\begin{eqnarray}
E'_{w}(\tau )\leq 0.\nonumber
\end{eqnarray}
 We now can obtain the next important Lemma.
\begin{lemma} \label{lem1}
Suppose that $v\in C^{1}(0,T,H_0^{1}(r_1,r_2)  )  \cap C^2(0,T,L^2(r_1,r_2)  )$ is a solution of the semi-linear wave equation
(\ref{e5}). Then for $\tau\geq0$,
\begin{eqnarray}
 E_{w}(\tau) \leq E_{w}(0)  -\frac{p-1}{p+1}\int_0^{\tau} e^{\frac{(p-1)s}{2}}\int_{r_1}^{r_2}\vert w\vert^{p+1}dxds, \label{e6}
\end{eqnarray}
\end{lemma}
\begin{proof} Taking the $L^2$ product of (\ref{e5}) with $w_{\tau}$ yields
	\begin{eqnarray}
	&&\int_{r_1}^{r_2}w_{\tau\tau}w_{\tau}dx-\int_{r_1}^{r_2
	}\Big( w_{xx} -e^{-\tau/2}\int_{0}^{t}e^{s/2}\mu(s) w_{xx}(t-s) ds\Big) w_{\tau}dx\nonumber\\\nonumber\\
	&&=\frac{1}{4}\int_{r_1}^{r_2}w w_{\tau}dx+\int_{r_1}^{r_2} e^{(p-1)\tau/2}w^p w_{\tau}dx. \nonumber
	\end{eqnarray}
	Thus, by Lemma \ref{lemma0} with $\alpha(\tau)=e^{-\tau /2}$, we have  
	\begin{eqnarray}
	&&\frac{1}{2}\frac{d}{d\tau} \Big[\int_{r_1}^{r_2}\vert w_{\tau}\vert^2 dx+ (1-\int_{0}^{t}e^{s/2}\mu(s)ds) \int_{r_1}^{r_2} \vert   w_{x}\vert^2 dx \Big]\nonumber\\\nonumber\\  &&+\frac{1}{2}\frac{d}{d\tau}\int_{0}^{\tau}e^{s/2}\mu(\tau-s)\int_{r_1}^{r_2} \left\vert   w_{x}(s)- w_{x}(\tau)\right\vert ^{2}dxds\nonumber\\\nonumber\\
	&&=\frac{1}{8}\frac{d}{d\tau}  \int_{r_1}^{r_2}\vert w\vert ^2dx+\frac{1}{p+1}\frac{d}{d\tau}\int_{r_1}^{r_2} e^{(p-1)\tau/2}w^{p+1} w_{\tau}dx+\frac{2(p-1)}{p+1}\int_{r_1}^{r_2} e^{(p-1)\tau/2}w^{p+1}dx. \nonumber\\\nonumber\\ 
	&&+\frac{1}{2}\alpha(\tau)\int_{0}^{\tau}\Big(e^{s/2}\mu(\tau-s)\Big)'\int_{r_1}^{r_2} \left\vert   w_{x}(s)- w_{x}(\tau)\right\vert ^{2}dxds \nonumber\\ \nonumber\\
	&&-\frac{1}{2}\mu(\tau)\int_{r_1}^{r_2} \left\vert w_{x}(t)\right\vert ^{2}dx \nonumber\\\nonumber\\
	&&+\frac{1}{2}\alpha'(\tau)\int_{0}^{\tau}e^{s/2}\mu(\tau-s)\int_{r_1}^{r_2} \left\vert   w_{x}(s)- w_{x}(\tau)\right\vert ^{2}dxds\nonumber\\ \nonumber\\
	&&-\frac{1}{2}\alpha'(\tau)\int_{0}^{s}e^{s/2}\mu(s)ds\int_{r_1}^{r_2} \left\vert w_{x}(\tau)\right\vert ^{2}dx.\nonumber
	\end{eqnarray}
	Then, by conditions on $\mu, \alpha$ and (\ref{energy}), the assertions  (\ref{e6}) is proved.
\end{proof}
\section{Blow up result for $E_u(0)=0$}
Under small amplitude initial
data, we prove that $w$ blows up in $L^2$  at finite time $\ln T^{\ast}$ in the following Theorem \ref{thm2}.

\begin{theorem} \label{thm2}
Suppose that $w  \in C^{1}(0,T,H_0^{1}(r_1,r_2)  )  \cap C^2(0,T,L^2(r_1,r_2)  )$ is a weak solution of equation (\ref{e5}) with
$$e(0):=\int_{r_1}^{r_2}u_0u_1(x)  dx>0,\qquad E_u(0)=0$$
and $0<r_2-r_1\leq 1$.
Then there
exists $T_1^{\ast}$ such that
\[
\int_{r_1}^{r_2} \vert u(t,x)\vert ^2dx \to +\infty
\quad\hbox{ as } t\to T_1^{\ast},
\]
where $$ T_1^{\ast}=\frac{2}{p-1}\Big( \int_{r_1}^{r_2} \vert u_0\vert dx\Big) \Big(\int_{r_1}^{r_2} u_0u_1dx\Big)^{-1}.$$
\end{theorem}
We need to state and prove the next intermediate Lemma.
\begin{lemma}\label{Le10}
	Suppose that $w $ is a weak solution of equation (\ref{e5}). Then
	\begin{eqnarray} 
	&&  \int_{r_1}^{r_2}e^{\frac{p-1}{2}s}w^{p+1}(s,x) dx \nonumber\\\nonumber\\
	&&  \geq \frac{p+1}{2}\Big[\int_{r_1}^{r_2}
	\vert w_{s}\vert ^2dx+(1-\int_{0}^{t}e^{s/2}\mu(s)ds)\int_{r_1}^{r_2} \vert w_{x}\vert^2dx-\frac{1}{4}\int_{r_1}^{r_2}\vert w\vert ^2dx \Big] \nonumber\\\nonumber\\
	&&+\int_{0}^{\tau}e^{s/2}\mu(\tau-s)\int_{r_1}^{r_2} \left\vert   w_{x}(s)- w_{x}(\tau)\right\vert ^{2}dxds-(p+1)  E_{w}(
	0)  e^{\frac{p-1}{2}s} \nonumber\\\nonumber\\
	&&\quad  +\frac{p^2-1}{2}\int_0^{s}e^{\frac{p-1}{2}(s-r)  }
	\Big[\int_{r_1}^{r_2}
	\vert w_{s}\vert ^2dx+(1-\int_{0}^{t}e^{s/2}\mu(s)ds)\int_{r_1}^{r_2} \vert w_{x}\vert^2dx-\frac{1}{4}\int_{r_1}^{r_2}\vert w\vert ^2dx \Big]dr\nonumber\\\nonumber\\
	&&+\frac{p^2-1}{2}\int_0^{s}e^{\frac{p-1}{2}(s-r)  }\int_{0}^{\tau}e^{s/2}\mu(\tau-s)\int_{r_1}^{r_2} \left\vert   w_{x}(s)- w_{x}(\tau)\right\vert ^{2}dxds.\nonumber
	\end{eqnarray}
\end{lemma}
\begin{proof}
Set
\begin{eqnarray}
L(s)  &=&\frac{1}{p+1}\int_0^{s}
e^{\frac{p-1}{2}r}\int_{r_1}^{r_2}\vert w\vert^{p+1}dx dr,\nonumber\\\nonumber\\
 F(s)   &=&\int_{r_1}^{r_2}\vert w_{s}\vert^2dx+(1-\int_{0}^{s}e^{\tau/2}\mu(\tau)d\tau)\int_{r_1}^{r_2} \vert w_{x}\vert 
^2dx\nonumber \\\nonumber\\
&-&\frac{1}{4}\int_{r_1}^{r_2}\vert w\vert^2dx+\int_{0}^{\tau}e^{s/2}\mu(\tau-s)\int_{r_1}^{r_2}\left\vert   w_{x}(s)- w_{x}(\tau)\right\vert ^{2}dxds,\nonumber
\end{eqnarray}
By Lemma \ref{lemma0} and Lemma \ref{lem1}, equation (\ref{e6})  can be rewritten as
\begin{eqnarray}
E_{w}(0)\geq F-2L'+(p-1)  L, \label{e8}
\end{eqnarray}
therefore,
\begin{eqnarray}
(e^{\frac{p-1}{-2}s}L)  '
&=&e^{\frac{p-1}{-2}s}\Big(L'-\frac{p-1}{2}L\Big)\nonumber
\\\nonumber\\
&\geq&\frac{1}{2}e^{\frac{p-1}{-2}s}(F-E_{w}(0)  ), \nonumber
\end{eqnarray}
and
\begin{eqnarray}
e^{\frac{p-1}{-2}s}L
&  \geq&\frac{1}{2}\int_0^{s}e^{\frac{p-1}{-2}r}(
F(r)  -E_{w}(0)  )  dr\nonumber\\ \nonumber
\\
& \geq&\frac{1}{2}\int_0^{s}e^{\frac{p-1}{-2}r}F(r)
dr-\frac{E_{w}(0)  }{p-1}\Big(1-e^{\frac{p-1}{-2}s}\Big),\nonumber
\end{eqnarray}
and
\begin{eqnarray}
L\geq\frac{1}{2}\int_0^{s}e^{\frac{p-1}{2}(s-r)  }F(
r)  dr-\frac{E_{w}(0)  }{p-1}\Big(e^{\frac{p-1}{2}s}-1\Big);\nonumber
\end{eqnarray}
this implies
\begin{eqnarray}
&&  \frac{1}{p+1}\int_0^{s} e^{\frac{p-1}{2}r}\int_{r_1}^{r_2}
\vert w\vert ^{p+1}dx  \,dr\nonumber\\\nonumber\\
&& \geq\frac{1}{2}\int_0^{s}e^{\frac{p-1}{2}(s-r) }\Big[
 \int_{r_1}^{r_2}\vert w_{s}\vert^2dx+(1-\int_{0}^{t}e^{s/2}\mu(s)ds)\int_{r_1}^{r_2} \vert w_{x}\vert^2dx-\frac{1}{4}\int_{r_1}^{r_2}\vert w\vert^2dx \Big]\,dr\nonumber\\\nonumber\\
 &&-\frac{E_{w}(0)  }{p-1}(e^{\frac{p-1}{2}s}-1)+\frac{1}{2}\int_0^{s}e^{\frac{p-1}{2}(s-r) }\int_{0}^{\tau}e^{s/2}\mu(\tau-s)\int_{r_1}^{r_2} \left\vert   w_{x}(s)- w_{x}(\tau)\right\vert ^{2}dxds, \nonumber
\end{eqnarray}
and
\begin{eqnarray} \label{e9}
& & \int_0^{s}\int_{r_1}^{r_2}e^{\frac{p-1}{2}r}w^{p+1}(r,x)  \,dx\,dr\nonumber\\\nonumber\\
&&  \geq\frac{p+1}{2}\int_0^{s}e^{\frac{p-1}{2}(s-r)  }\Big[
\int_{r_1}^{r_2}\vert w_{s}\vert^2dx+(1-\int_{0}^{t}e^{s/2}\mu(s)ds)\int_{r_1}^{r_2} \vert w_{x}\vert^2dx-\frac{1}{4}\int_{r_1}^{r_2}\vert w\vert^2dx \Big]dr\nonumber\\\nonumber\\
&&  -\frac{p+1}{p-1}E_{w}(0)  (e^{\frac{p-1}{2}s}-1)+\frac{p+1}{2}\int_0^{s}e^{\frac{p-1}{2}(s-r) }\int_{0}^{\tau}e^{s/2}\mu(\tau-s)\int_{r_1}^{r_2} \left\vert   w_{x}(s)- w_{x}(\tau)\right\vert ^{2}dxds, \nonumber
\end{eqnarray}
and
\begin{eqnarray} 
&&  \int_{r_1}^{r_2}e^{\frac{p-1}{2}s}w^{p+1}(s,x) dx \nonumber\\\nonumber\\
&&  \geq\frac{p+1}{2}\Big[\int_{r_1}^{r_2}
\vert w_{s}\vert ^2dx+(1-\int_{0}^{t}e^{s/2}\mu(s)ds)\int_{r_1}^{r_2} \vert w_{x}\vert^2dx-\frac{1}{4}\int_{r_1}^{r_2}\vert w\vert ^2dx \Big] \nonumber\\\nonumber\\
&&+\int_{0}^{\tau}e^{s/2}\mu(\tau-s)\int_{r_1}^{r_2} \left\vert   w_{x}(s)- w_{x}(\tau)\right\vert ^{2}dxds-(p+1)  E_{w}(
0)  e^{\frac{p-1}{2}s} \nonumber\\\nonumber\\
&&\quad  +\frac{p^2-1}{2}\int_0^{s}e^{\frac{p-1}{2}(s-r)  }
\Big[\int_{r_1}^{r_2}
\vert w_{s}\vert ^2dx+(1-\int_{0}^{t}e^{s/2}\mu(s)ds)\int_{r_1}^{r_2} \vert w_{x}\vert^2dx-\frac{1}{4}\int_{r_1}^{r_2}\vert w\vert ^2dx \Big]dr\nonumber\\\nonumber\\
&&+\frac{p^2-1}{2}\int_0^{s}e^{\frac{p-1}{2}(s-r)  }\int_{0}^{\tau}e^{s/2}\mu(\tau-s)\int_{r_1}^{r_2} \left\vert   w_{x}(s)- w_{x}(\tau)\right\vert ^{2}dxds.\label{A}
\end{eqnarray}
This completes the proof.
\end{proof} 
We are now ready to prove Theorem \ref{thm2}
\begin{proof}(Of Theorem \ref{thm2})\\
Let $$A(s)  :=\int_{r_1}^{r_2}\vert w\vert ^2dx,$$
then we have $$A'(s)  =2\int_{r_1}^{r_2}ww_{s}(s,x)  dx.$$
and
\begin{eqnarray}
A''(s)
&=&2\int_{r_1}^{r_2}ww_{ss}(s,x)  dx+2\int_{r_1}^{r_2}w^2_{s}(s,x)  dx\nonumber\\\nonumber\\
&=&2\int_{r_1}^{r_2}(ww_{xx}-we^{-\tau/2}\int_{0}^{t}e^{s/2}\mu(s)w_{xx}(t-s)ds+\frac{1}{4}w^2+w_{s}^2
+e^{\frac{p-1}{2}s}w^{p+1})  dx\nonumber\\\nonumber\\
& =&2\int_{r_1}^{r_2}(-w_{x}^2+w_xe^{-\tau/2}\int_{0}^{t}e^{s/2}\mu(s)w_{x}(t-s)ds+\frac{1}{4}w^2+w_{s}
^2+e^{\frac{p-1}{2}s}w^{p+1})dx.\nonumber
\end{eqnarray}
By Lemma\ref{lemma0}, Lemmad\ref{Le10} and (\ref{A}), then
\begin{eqnarray}
  A''(s) 
&\geq&2 \Big((\int_0^te^{s/2}\mu(s)ds-1)\int_{r_1}^{r_2} \vert w_{x}\vert^2dx +\frac{1}{4}\int_{r_1}^{r_2}\vert w\vert^2dx+\int_{r_1}^{r_2}\vert w_{s}\vert^2dx\Big) \nonumber\\ \nonumber\\
&-&2\int_{0}^{\tau}e^{s/2}\mu(\tau-s)\int_{r_1}^{r_2} \left\vert   w_{x}(s)- w_{x}(\tau)\right\vert ^{2}dxds\nonumber\\ \nonumber\\
&+&(p+1)2 \Big((\int_0^te^{s/2}\mu(s)ds-1)\int_{r_1}^{r_2} \vert w_{x}\vert^2dx +\frac{1}{4}\int_{r_1}^{r_2}\vert w\vert^2dx+\int_{r_1}^{r_2}\vert w_{s}\vert^2dx\Big) \nonumber\\ \nonumber\\
&-&(p+1)\int_{0}^{\tau}e^{s/2}\mu(\tau-s)\int_{r_1}^{r_2} \left\vert   w_{x}(s)- w_{x}(\tau)\right\vert ^{2}dxds\nonumber \\\nonumber  \\
&+&(p^2-1)  \int_0^{s}e^{\frac{p-1}{2}(s-r)} \Big((\int_0^te^{s/2}\mu(s)ds-1)\int_{r_1}^{r_2} \vert w_{x}\vert^2dx +\frac{1}{4}\int_{r_1}^{r_2}\vert w\vert^2dx+\int_{r_1}^{r_2}\vert w_{s}\vert^2dx\Big) \nonumber\\ \nonumber\\
&-&(p^2-1)  \int_0^{s}e^{\frac{p-1}{2}(s-r)}\int_{0}^{\tau}e^{s/2}\mu(\tau-s)\int_{r_1}^{r_2} \left\vert   w_{x}(s)- w_{x}(\tau)\right\vert ^{2}dxds
 \nonumber \\\nonumber  \\
&-&2(p+1)  E_{w}(0) e^{\frac{p-1}{2}s} \nonumber  \\\nonumber  \\
&  \geq& \big[(p+3)  \int_{r_1}^{r_2} \vert w_{s}\vert^2dx+(
p-1)(1-\int_0^te^{s/2}\mu(s)ds)  \int_{r_1}^{r_2} \vert w_{x}\vert^2dx-\frac{p-1}{4}\int_{r_1}^{r_2}\vert w\vert^2dx\big]  \nonumber\\\nonumber\\
&-&2(p+1)  E_{w}(0)  e^{\frac{p-1}{2}s} +(p-1)\int_{0}^{\tau}e^{s/2}\mu(\tau-s)\int_{r_1}^{r_2} \left\vert   w_{x}(s)- w_{x}(\tau)\right\vert ^{2}dxds\\\nonumber\\
\label{e11}
&+&(p^2-1)  \int_0^{s}e^{\frac{p-1}{2}(s-r)
} \Big(\int_{r_1}^{r_2}\vert w_{s}\vert^2dx+(\int_0^te^{s/2}\mu(s)ds-1)\int_{r_1}^{r_2} \vert w_{x}\vert^2dx+\frac{1}{4}\int_{r_1}^{r_2}\vert w\vert^2dx\Big) \,dr\nonumber\\ \nonumber\\
&-&(p^2-1)  \int_0^{s}e^{\frac{p-1}{2}(s-r)
}\int_{0}^{\tau}e^{s/2}\mu(\tau-s)\int_{r_1}^{r_2} \left\vert   w_{x}(s)- w_{x}(\tau)\right\vert ^{2}dxdsdr.
\nonumber
\end{eqnarray}
As in \cite{L11}, let us setting $$J(s)  :=A(s)  ^{-k}, \qquad k=\frac{p-1} {4}>0.$$
Then  $$J'(s)  =-kA(s)  ^{-k-1}A'(s),$$
and
\begin{eqnarray} \label{e12}
J''(s)   &  =-kA(s)  ^{-k-2}[A(s)  A''(s)  -(k+1) A'(s)  ^2] \nonumber\\\nonumber\\
&  \leq-kA(s)  ^{-k-1}\big[A''(s)
-4(k+1)  \int_{r_1}^{r_2}w_{s}^2  dx\big].
\end{eqnarray}
Since $E_{u}(0)=0$, we have
\begin{eqnarray}
& & A''(s)  -4(k+1)  \int_{r_1}^{r_2}\vert w_{s}\vert ^2dx \nonumber\\\nonumber\\
&&  \geq\Big[(p+3) \int_{r_1}^{r_2}\vert w_{s}\vert^2dx+(
p-1)(1-\int_0^te^{s/2}\mu(s)ds) \int_{r_1}^{r_2} \vert w_{x}\vert^2dx-\frac{p-1}{4}\int_{r_1}^{r_2}\vert w\vert^2dx\Big]   \nonumber\\\nonumber\\
&&  +(p^2-1)  \int_0^{s}e^{\frac{p-1}{2}(s-r)
} \Big(\int_{r_1}^{r_2}\vert w_{s}\vert^2dx+(1-\int_0^te^{s/2}\mu(s)ds)\int_{r_1}^{r_2} \vert w_{x}\vert^2dx-\frac{1}{4}\int_{r_1}^{r_2}\vert w\vert^2dx\Big) \,dr\nonumber\\\nonumber\\
& &\quad -4(k+1)  \int_{r_1}^{r_2}\vert w_{s}\vert^2dx+(p-1)\int_{0}^{\tau}e^{s/2}\mu(\tau-s)\int_{r_1}^{r_2} \left\vert   w_{x}(s)- w_{x}(\tau)\right\vert ^{2}dxds\nonumber\\
&&+(p^2-1)  \int_0^{s}e^{\frac{p-1}{2}(s-r)
}\int_{0}^{\tau}e^{s/2}\mu(\tau-s)\int_{r_1}^{r_2} \left\vert   w_{x}(s)- w_{x}(\tau)\right\vert ^{2}dxdsdr,\nonumber
\end{eqnarray}
then,
\begin{eqnarray}
&&  A''(s)  -4(k+1)  \int_{r_1}^{r_2}\vert w_{s}\vert^2dx\nonumber\\\nonumber\\
&& \geq(p-1)  \Big[(1-\int_0^te^{s/2}\mu(s)ds)\int_{r_1}^{r_2} \vert w_{x}\vert^2dx-\frac{1}{4}\int_{r_1}^{r_2}\vert w\vert^2dx\Big]  \nonumber\\\nonumber\\
&&+(p-1)\int_{0}^{\tau}e^{s/2}\mu(\tau-s)\int_{r_1}^{r_2} \left\vert   w_{x}(s)- w_{x}(\tau)\right\vert ^{2}dxds \nonumber\\
& & +(p^2-1)  \int_0^{s}e^{\frac{p-1}{2}(s-r)}
\Big(\int_{r_1}^{r_2}\vert w_{s}\vert^2dx+(1-\int_0^te^{s/2}\mu(s)ds)\int_{r_1}^{r_2} \vert w_{x}\vert^2dx-\frac{1}{4}\int_{r_1}^{r_2}\vert w\vert^2dx\Big) dr\nonumber\\\nonumber\\
&&+(p^2-1)  \int_0^{s}e^{\frac{p-1}{2}(s-r)}\int_{0}^{\tau}e^{s/2}\mu(\tau-s)\int_{r_1}^{r_2} \left\vert   w_{x}(s)- w_{x}(\tau)\right\vert ^{2}dxdsdr \nonumber\\ \nonumber\\
& & \geq(p-1)  \big(1-(r_2-r_1)  ^2\big)
\Big(\int_{r_1}^{r_2} \vert w_{x}\vert^2dx+\int_{0}^{\tau}e^{s/2}\mu(\tau-s)\int_{r_1}^{r_2} \left\vert   w_{x}(s)- w_{x}(\tau)\right\vert ^{2}dxds\Big)\nonumber\\\nonumber\\
&&+(p+1)  \int_0^{s}e^{\frac{p-1}{2}(s-r)  }
\Big(\int_{r_1}^{r_2} \vert w_{s}\vert ^2dx+(1-\int_0^te^{s/2}\mu(s)ds)\int_{r_1}^{r_2} \vert w_{x}\vert ^2dx\Big)dr \nonumber\\
&&  +(p+1)  \int_0^{s}e^{\frac{p-1}{2}(s-r)  }\int_{0}^{\tau}e^{s/2}\mu(\tau-s)\int_{r_1}^{r_2} \left\vert   w_{x}(s)- w_{x}(\tau)\right\vert ^{2}dxdsdr>0,\nonumber
\end{eqnarray}
where $r_2\leq1+r_1$.\\
Therefore, by (\ref{e12}) we obtain that for,
 $r_2 -r_1\leq1$, $J''(s)  <0$ for all $s\geq0$.\\
\begin{eqnarray}
J'(s)   \leq J'(0)  &=&-\frac
{p-1}{4}A(0)  ^{-\frac{p+3}{4}}A'(0)\nonumber\\\nonumber\\
&=&-\frac{p-1}{2}e(0)\int_{r_1}^{r_2}\vert u_0\vert ^{-(p+3)}dx,\nonumber
\end{eqnarray}
and
\begin{eqnarray}
J(s)     &\leq& J(0)  -\frac{p-1}{2}e(0)\int_{r_1}^{r_2}\vert u_0\vert ^{-(p+3)}dxs\nonumber\\ \nonumber\\
&=&\int_{r_1}^{r_2}\| u_0\| ^{-(p-1)}dx-\frac{p-1}{2}e(0)\int_{r_1}^{r_2}\vert u_0\vert ^{-(p+3)}dxs\nonumber\\ \nonumber\\
  &=&\int_{r_1}^{r_2}\vert u_0\vert ^{-(p+3)}dx\Big(\int_{r_1}^{r_2}\vert u_0\vert dx-\frac{p-1}{2}e(0) s\Big). \nonumber
\end{eqnarray}
Then
\begin{eqnarray}
J(s)  \to 0 \quad \hbox{ as  }s\to T^{\ast}=\frac{2}{p-1}\frac{\int_{r_1}^{r_2}\vert u_0\vert dx}{e(0)}.
\end{eqnarray}
Thus $w$ solution of (\ref{e5}) blows up in $L^2$  at finite time $T^{\ast}$.
\end{proof}
\section{Blow up result for $E_u(0)<0$}
In the following theorem we shall state and prove our second blowing up result
\begin{theorem} \label{thm3}
Suppose that $w\in C^{1}(0,T,H_0^{1}(r_1,r_2)  )  \cap C^2(0,T,L^2(r_1,r_2)  )$ is a weak solution of equation (\ref{e1}) with
$$e(0)=\int_{r_1}^{r_2}u_0 u_1(x)  dx>0,\qquad E_u(0)<0,$$ and $0<r_2-r_1\leq1$.
Then, there exists $T_2^{\ast}$ such that
\[
\frac{1}{\int_{r_1}^{r_2}\vert u(t,x)\vert ^{2}dx} 
\to0\quad\hbox{ as }  t\to\ln T_2^{\ast}.
\]
Further, we have
$\ln T_2^{\ast}<\ln T_1^{\ast}$, and the estimate
\[
\int_{r_1}^{r_2}w^2   dx  \geq \int_{r_1}^{r_2}u^2_0 dx  -2E_{u}(0)  \frac
{p+1}{p-1}\big[se^{\frac{p-1}{2}s}-\frac{2}{p-1}(e^{\frac{p-1}{2}
s}-1)  \big].
\] 
\end{theorem}

\begin{proof}
By (\ref{e11}), Lemma\ref{lemma0}, $E_{u}(0)  <0$, $e(0)>0$ and $0<r_2-r_1\leq1$, we have \\
\begin{eqnarray} \label{e13}
J''(s) &\leq&-k\Big(\int_{r_1}^{r_2}w^2   dx\Big) ^{-k-1}\Big[A''(s)
-(p+3)  \int_{r_1}^{r_2}w_{s}^2(s,x) dx\Big] \nonumber\\\nonumber\\
&  =&-k\Big(\int_{r_1}^{r_2}w^2   dx\Big) ^{-k-1}\Big[
-2(p+1)  E_{w}(0)  e^{\frac{p-1}{2}s}\nonumber\\\nonumber\\
&+&(p-1)   \Big((1-\int_{0}^{t}e^{s/2}\mu(s)ds)\int_{r_1}^{r_2} \vert w_{x}\vert ^2dx-\frac{1}{4}\int_{r_1}^{r_2}\vert w\vert^2dx\Big)   \nonumber\\\nonumber\\
&+&(p-1)\int_{0}^{\tau}e^{s/2}\mu(\tau-s)\int_{r_1}^{r_2} \left\vert   w_{x}(s)- w_{x}(\tau)\right\vert ^{2}dxds \nonumber\\\nonumber\\
&+&(p^2-1)  \int_0^{s}e^{\frac{p-1}{2}(s-r)  }
 \Big(\int_{r_1}^{r_2}\vert w_{s}\vert^2dx+(1-\int_{0}^{t}e^{s/2}\mu(s)ds)\int_{r_1}^{r_2} \vert w_{x}\vert^2dx-\frac{1}{4}\int_{r_1}^{r_2}\vert w\vert^2dx\Big) dr \Big] \nonumber\\\nonumber\\
 &+&(p^2-1)  \int_0^{s}e^{\frac{p-1}{2}(s-r)  }\int_{0}^{\tau}e^{s/2}\mu(\tau-s)\int_{r_1}^{r_2} \left\vert   w_{x}(s)- w_{x}(\tau)\right\vert ^{2}dxdsdr\nonumber\\\nonumber\\
&  \leq& 2k(p+1)  E_{u}(0)  e^{\frac{p-1}{2}s}J(s)  ^{1+\frac{1}{k}}<0,
\end{eqnarray}
where $k=(p-1)  /4$, we can obtain the same conclusions as in
Theorem \ref{thm2}.\\
By the inequality (\ref{e13}) and $J'<0$ we can estimate
$J$ further,
\begin{eqnarray}
J''(s)
&\leq&2k(p+1)  E_{u}(0)  e^{\frac{p-1}{2}s}J(s)  ^{1+\frac{1}{k}} \nonumber\\\nonumber\\
&=&\frac{1}{2}(p^2-1)  E_{u}(0)  e^{\frac{p-1}{2}s}J(s)  ^{1+\frac{1}{k}}<0, \nonumber
\end{eqnarray}
and
\begin{eqnarray}
J'(s)  &\leq& J'(0)  +\frac{s}
{2}(p^2-1)  E_{u}(0)  e^{\frac{p-1}{2}s}J(
s)  ^{1+\frac{1}{k}}\nonumber\\\nonumber\\
&\leq&\frac{s}{2}(p^2-1)  E_{w}(
0)  e^{\frac{p-1}{2}s}J(s)  ^{1+\frac{1}{k}},\nonumber
\end{eqnarray}
and
\begin{eqnarray}
-k\big(J(s)  ^{-\frac{1}{k}}\big)'
&=&J(s)  ^{-1-\frac{1}{k}}J'(s)  \nonumber\\\nonumber\\
&\leq&\frac{E_{u}(0)  }{2}(p^2-1)  se^{\frac{p-1}{2}s},\nonumber
\end{eqnarray}
and
\begin{eqnarray}
-k(J(s)  ^{-\frac{1}{k}}-J(0)  ^{-\frac{1}{k}})
&  \leq&\frac{E_{u}(0)  }{2}(p^2-1)
 \Big(\frac{2}{p-1}se^{\frac{p-1}{2}s}-(\frac{2}{p-1}) ^2(e^{\frac{p-1}{2}s}-1) \Big) \nonumber
 \\\nonumber\\
& =&E_{w}(0) (p+1) \big[se^{\frac{p-1}{2}
s}-\frac{2}{p-1}(e^{\frac{p-1}{2}s}-1)  \big],\nonumber
\end{eqnarray}
which implies
\[
 \int_{r_1}^{r_2}w^2   dx \geq  \int_{r_1}^{r_2}u_0^2   dx  -2\frac{p+1}{p-1}E_{u}(
0)  \big[se^{\frac{p-1}{2}s}-\frac{2}{p-1}(e^{\frac{p-1}{2}
s}-1)  \big]   
\]Then $u$ solution of our initial problem (\ref{e1}) blows up in
$L^2$ \ at finite time $\ln T_2^{\ast}$. This completes the proof. 
\end{proof}


\begin{thebibliography}{00}




\bibitem{BENAISSA} A. Benaissa, D. Ouchenane and Kh. Zennir, \emph{Blow up of positive initial-energy solutions to systems of nonlinear wave equations with degenerate damping and source terms}, Nonlinear studies. Vol. 19, No. 4, pp. 523-535, 2012.

\bibitem{B2} F. E. Browder;
\emph{On non-linear wave equations}. M.Z. 80. pp. 249-264 (1962).


\bibitem{Cavalcanti}M.M. Cavalcanti, L.H. Fatori and T.F. Ma;
\emph{Attractors for wave equations with degenerate memory}, J. Diff. Eq., 260 (2016), pp. 56-83.


\bibitem{C1} S. Chandrasekhar;
\emph{Introduction to the Study of Stellar Structure}, Chap.
4. Dover, New York, 1957


\bibitem{Dafermos} C.M. Dafermos. An abstract Volterra equation with applications to linear viscoelasticity. J. Diff. Equations, 7 (1970), 554-569.
\bibitem{NDafermos} Dafermos C. M., On the existence and the asymptotic stability of solution to the equations of linear thermoelasticity. Arch. Ration. Mech. Anal., 29, (1968) 241-271.

\bibitem{C. M. Dafermos} C. M. Dafermos, H.P. Oquendo,  Asymptotic stability in viscoelasticity. Arch. Ration. Mech. Anal. 37(1970), 297-308.

\bibitem{C4} Conti, Graffi, G. Sansone;
\emph{Qualitative Methods in the Theory of Nonlinear Vibrations},
Proc. Internat. Sympos. Nonlinear Vibrations, vol. II, 1961, pp. 172-189.

\bibitem{E1} R. Emden, Gaskugeln;
\emph{Anwendungen der mechanischen Warmetheorie auf
Kosmologie und meteorologische Probleme}, B. G.Teubner, Leipzig, Germany 1907.

\bibitem{F1} R. H. Fowler;
\emph{The form near infinity of real, continuous solutions of a
certain differential equation of the second order}, Quart. J. Math., 45 (1914), pp. 289-350.

\bibitem{F2} R. H. Fowler;
\emph{The solution of Emden's and similar differential equations},
Monthly Notices Roy. Astro. Soc., 91 (1930), pp. 63-91.

\bibitem{F3} R. H. Fowler;
\emph{Some results on the form near infinity of real continuous
solutions of a certain type of second order differential equations}, Proc.
London Math. Soc., 13 (1914), pp. 341-371.

\bibitem{F4} R. H. Fowler;
\emph{Further studies of Emden's and similar differential
equations}, Quart. J. Math., 2 (1931), pp. 259-288.

\bibitem{G1} R. Glassey;
\emph{Finite-time blow-up for solutions of nonlinear wave equations}.
M. Z. 177 (1981), pp. 323-340.

\bibitem{J2} F. John;
\emph{Blow-up for quasilinear wave equations in three space
dimensions}. Comm.Pure. Appl. Math. 36 (1981) pp. 29-51.


\bibitem{H2} M. L. J. Hautus;
\emph{Uniformly asymptotic formulas for the Emden-Fowler
differential equation}, J. Math. Anal. Appl., 30 (1970), pp. 680-694.

\bibitem{K1} S. Klainerman;
\emph{Global existence for nonlinear wave equations}. Comm.Pure
Appl. Math. 33 (1980), pp. 43-101.

\bibitem{K2} S. Klainerman, G. Ponce;
\emph{Global, small amplitude solutions to nonlinear
evolution equations}. Comm. Pure Appl. Math. 36 (1983), pp. 133-141.


\bibitem{L11} M. R. Li;
\emph{Nonexistence of global solutions of Emden-Fowler type
	semilinear wave equations with non-positive energy}. Electronic Journal of Differential Equations,
Vol. 2016 (2016), No. 93, pp. 1--10.\\

\bibitem{L2} M. R. Li;
\emph{Estimates for the life-span of the solutions of some
semilinear wave equations}. ACPAA. Vol. 7 (2008), No. 2, pp. 417-432.

\bibitem{L4} M. R. Li;
\emph{Existence and uniqueness of local weak solutions for the
Emden--Fowler wave equation in one dimension}, Electronic Journal of
Differential Equations, Vol. 2015 (2015), No. 145, pp. 1--10.

\bibitem{Ouchenane} D. Ouchenane,	Kh. Zennir and M. Bayoud, \emph{Global nonexistence of solutions for a system of nonlinear viscoelastic wave equations with degenerate damping and source terms}, Ukrainian Mathematical Journal, Vol. 65, No. 7(2013), 723-739.

\bibitem{R1} R. Racke;
\emph{Lectures on nonlinear Evolution Equations: Initial Value
Problems}. Aspects of Math.  Braunschweig Wiesbaden Vieweg(1992).


 \bibitem{R2} A. Ritter;
 \emph{Untersuchungen \"{u}ber die H\"{o}he der Atmosph\"{a}re
 	und\ die Konstitution gasformiger Weltk\"{o}rper, 18 articles}, Wiedemann
 Annalen der Physik, 5-20, pp. 1878-1883.

\bibitem{S1} W. A. Strauss;
\emph{Nonlinear Wave Equations}, AMS Providence(1989).
Dimensions. J. Differential Equations 52 (1984), pp.378-406.


\bibitem{S3} T. Sideris;
\emph{Nonexistence of global solutions to semilinear wave equations
in high dimensions}. J. Differential Equations 52(1982). pp. 303-345.


 \bibitem{T1} W. Thompson (Lord Kelvin);
 \emph{On the convective equilibrium of temperature
 	in the atmosphere}, Manchester Philos. Soc. Proc., 2 (1860-62), pp.170-176;
 reprint, Math. and Phys. Papers by Lord Kelvin, 3 (1890), pp. 255-260.

\bibitem{zennirAMEN} Kh. Zennir and A. Guesmia;
\emph{Existence of solutions to nonlinear kth-order coupled Klein-Gordon equations with nonlinear sources and memory term}. Applied Mathematics E-Notes, 15(2015), 121-136.

\bibitem{zennirvladiv} Kh. Zennir and S. Zitouni;
\emph{On the absence of solutions to damped system of nonlinear wave equations of Kirchhoff-type}. Vladikavkaz Mathematical Journal, 17(4), (2015), 44-58.

\end{thebibliography}
\end{document}